\documentclass[12pt, oneside]{amsart}   	
\usepackage{geometry}                		
\usepackage{graphicx}				
\usepackage{amssymb, amsmath, amsthm, proof, tikz, setspace, changepage}
\usetikzlibrary{matrix}
\usetikzlibrary{positioning,arrows,calc}
\tikzset{
modal/.style={>=stealth,shorten >=1pt,shorten <=1pt,auto,node distance=1.5cm,
semithick},
world/.style={circle,draw,minimum size=0.5cm,fill=gray!15},
point/.style={circle,draw,inner sep=0.5mm,fill=black},
reflexive above/.style={->,loop,looseness=7,in=120,out=60},
reflexive below/.style={->,loop,looseness=7,in=240,out=300},
reflexive left/.style={->,loop,looseness=7,in=150,out=210},
reflexive right/.style={->,loop,looseness=7,in=30,out=330}
}

\theoremstyle{definition}
\newtheorem{theorem}{Theorem}
\newtheorem{lemma}{Lemma}

\newtheorem{corollary}{Corollary}
\newtheorem{subclaim}{Subclaim}

\newenvironment{subproof}[1][\proofname]{%
  \begin{proof}[#1]%
}{%
  \end{proof}%
}

\title{A Canonical Model for Constant Domain Basic First-Order Logic} 
\author{Ben Middleton \\ University of Notre Dame}
\date{}

\begin{document}
\maketitle

\begin{abstract}
I build a canonical model for constant domain basic first-order logic ($\textsf{BQL}_\textsf{CD}$), the constant domain first-order extension of Visser's basic propositional logic, and use the canonical model to verify that $\textsf{BQL}_\textsf{CD}$ satisfies the disjunction and existence properties.
\end{abstract}

\section{Introduction}

Basic propositional logic (\textsf{BPL}) is the subintuitionistic propositional logic obtained by dropping the requirement on the Kripke models for intuitionistic propositional logic (\textsf{IPL}) that the accessibility relation is reflexive. Most notably, dropping reflexivity invalidates modus ponens.  Visser [7] introduced \textsf{BPL} and proved completeness by building a canonical model.  Basic first-order logic (\textsf{BQL}), introduced by Ruitenburg [6], is the subintuitionistic first-order extension of \textsf{BPL} obtained by (i) dropping the requirement on the Kripke models for intuitionistic first-order logic (\textsf{IQL}) that the accessibility relation is reflexive, (ii) restricting the object language by replacing the clause $\phi : \forall v \phi$ in the inductive definition of formulas with the clause $\phi, \psi : \forall \overline{v}(\phi \rightarrow \psi)$ and (iii) setting $w \Vdash \forall \overline{v}(\phi \rightarrow \psi)(\overline{a})$ iff for every $u \succ w$ and all $\overline{b} \in dom(u)$: $u \Vdash \phi(\overline{a}, \overline{b})$ only if $u \Vdash \psi(\overline{a}, \overline{b}).$\begin{footnote}{This definition of \textsf{BQL} differs from the definition given in Ruitenburg's original paper [6]. We identify \textsf{BQL} with the set of pairs $\langle \Gamma, \phi \rangle$ such that for every world $w$ in every \textsf{BQL} model: $w \Vdash \Gamma$ only if $w \Vdash \phi$ (where $\Gamma \cup \{\phi\}$ is a set of sentences in the language of \textsf{BQL}). By contrast, Ruitenburg identifies \textsf{BQL} with the set of sequents $\phi \hspace{-1mm} \implies \hspace{-1mm} \psi$ such that for every world $w$ in every \textsf{BQL} model: $w \Vdash \phi$ only if $w \Vdash \psi$ (where $\phi$ and $\psi$ are sentences in the language of \textsf{BQL}). So our definition of \textsf{BQL} essentially generalizes Ruitenburg's definition to allow for arbitrarily many (including zero) premises.}\end{footnote}  A natural way of extending \textsf{BQL} to the full object language was suggested by Restall [5], who advocated restricting the \textsf{BQL} models to those with constant domains and giving $\forall$ its classical satisfaction condition.\begin{footnote}{Restall also added an actual world and required the actual world see itself, thereby regaining modus ponens. We do not take this approach here.}\end{footnote} This results in constant domain basic first-order logic ($\textsf{BQL}_\textsf{CD}$).  The relationship between $\textsf{BQL}, \textsf{BQL}_\textsf{CD}$ and their intuitionistic counterparts is pictured below:\begin{footnote}{We identify a logic \textsf{L} with the set of pairs $\langle \Gamma, \phi \rangle$ such that $\Gamma \models_\textsf{L} \phi$.}\end{footnote}
$$
\begin{tikzpicture}
  \matrix (m) [matrix of math nodes, row sep=2em, column sep=2em, minimum width=2em]
  {
     \textsf{BQL}_\textsf{CD} & \textsf{IQL}_\textsf{CD} \\
     \textsf{BQL} & \textsf{IQL} \\};
  \path[-stealth]
    (m-1-1) edge [draw=none] node [] {$\subset$} (m-1-2)
    (m-2-1) edge [draw=none] node [sloped, auto=false, allow upside down] {$\subset$} (m-1-1)
    (m-2-1) edge [draw=none] node [] {$\subset$} (m-2-2)
    (m-2-2) edge [draw=none] node [sloped, auto=false, allow upside down] {$\subset$} (m-1-2);
\end{tikzpicture}
$$
The proper inclusion $\textsf{BQL} \subset \textsf{BQL}_\textsf{CD}$ holds even when we restrict $\textsf{BQL}_\textsf{CD}$ to the language of $\textsf{BQL}$, as witnessed by the fact that the inference
$$
\infer[(v \text{ not free in } \phi)]
	{\top \rightarrow \phi \vee \forall v(\top \rightarrow \psi)}
	{\forall v(\top \rightarrow \phi \vee \psi)}
$$
is valid in $\textsf{BQL}_\textsf{CD}$ but not in $\textsf{BQL}$. Restall attempted to prove completeness for $\textsf{BQL}_\textsf{CD}$ by building a canonical model but was unable to do so. More recently, Ishigaki and Kikuchi [2] proved completeness for $\textsf{BQL}_\textsf{CD}$ using a tree-sequent calculus.  However, a canonical model proof of completeness would be preferable, since we could use the canonical model to verify that $\textsf{BQL}_\textsf{CD}$ satisfies the disjunction and existence properties. The major impediment to generalizing Visser's canonical model construction for $\textsf{BPL}$ to $\textsf{BQL}_\textsf{CD}$ is ensuring the existence of witnesses when extending prime saturated $\textsf{BQL}_\textsf{CD}$-theories to larger prime saturated $\textsf{BQL}_\textsf{CD}$-theories.  Ordinarily, the existence of witnesses is guaranteed by adding fresh constant symbols to the object language.  However, this would result in a Kripke model with varying domains. The same problem arises when trying to generalize the canonical model construction for $\textsf{IPL}$ to $\textsf{IQL}_\textsf{CD}$. In the case of $\textsf{IQL}_\textsf{CD}$, a solution has been found (Gabbay, Shehtman and Skvortsov [1]).  In this paper, I build a canonical model for $\textsf{BQL}_\textsf{CD}$ by carrying over the solution for $\textsf{IQL}_\textsf{CD}$. I use the canonical model to verify that $\textsf{BQL}_\textsf{CD}$ satisfies the disjunction and existence properties.

\section{Constant Domain Basic First-Order Logic}

\subsection{Model Theory}

Let $\mathcal{L}$ be a first-order language with primitive operators $\top$, $\bot$, $\wedge$, $\vee$, $\rightarrow$, $\forall$, $\exists$. A transitive frame is a pair $\langle W, \prec \rangle$ such that $W$ is a non-empty set (the set of worlds) and $\prec$ is a transitive binary relation on $W$ (the accessibility relation).  An $\mathcal{L}$-model (for $\textsf{BQL}_\textsf{CD}$) is a $4$-tuple $\mathfrak{M} = \langle W, \prec, M, |\mathord{\cdot}| \rangle$ such that $\langle W, \prec \rangle$ is a transitive frame, $M$ is a non-empty set (the domain of quantification) and $|\mathord{\cdot}|$ is a function whose domain is the signature of $\mathcal{L}$ such that $|c| \in M$, $|f^n|: M^n \rightarrow M$ and $|R^n|: W \rightarrow \mathcal{P}(M^n)$, subject to the constraint that $w \prec u$ only if $|R^n|(w) \subseteq |R^n|(u)$. For a term $t(\overline{v}) \in \mathcal{L}$ and $\overline{a} \in M^n$, we recursively define $|t|(\overline{a})$ as follows:
\begin{align*}
|c|(\overline{a}) &= |c| \\
|v_i|(\overline{a}) &= a_i \\
|f^n(t_1,...,t_n)|(\overline{a}) &= |f^n|(|t_1|(\overline{a}),...,|t_n|(\overline{a})).
\end{align*}
For a formula $\phi(\overline{v}) \in \mathcal{L}, \overline{a} \in M^n$ and $w \in W$, we inductively define $\mathfrak{M}, w \Vdash \phi(\overline{a})$ as follows (suppressing $\mathfrak{M}$ for brevity):
\begin{align*}
w &\Vdash \top(\overline{a}) \\
w \Vdash R^n(t_1,...,t_n)(\overline{a}) &\iff \langle |t_1|(\overline{a}),...,|t_n|(\overline{a}) \rangle \in |R^n|(w) \\
w \Vdash (\phi \wedge \psi)(\overline{a}) &\iff w \Vdash \phi(\overline{a}) \text{ and } w \Vdash \psi(\overline{a}) \\
w \Vdash (\phi \vee \psi)(\overline{a}) &\iff w \Vdash \phi(\overline{a}) \text{ or } w \Vdash \psi(\overline{a}) \\
w \Vdash (\phi \rightarrow \psi)(\overline{a}) &\iff \text{for all } u \succ w: \text{ if } u \Vdash \phi(\overline{a}) \text{ then } u \Vdash \psi(\overline{a}) \\
w \Vdash \exists v \phi(\overline{a}) &\iff \text{for some } b \in M: w \Vdash \phi(\overline{a}, b)  \\
w \Vdash \forall v \phi(\overline{a}) &\iff \text{for all } b \in M: w \Vdash \phi(\overline{a}, b). 
\end{align*}
It follows by omission that $w \not\Vdash \bot(\overline{a})$.

\begin{theorem}[Persistence] If $w \Vdash \phi(\overline{a})$ and $w \prec u$ then $u \Vdash \phi(\overline{a})$.
\end{theorem}
\begin{proof}
An easy induction on the construction of $\mathcal{L}$-formulas.
\end{proof}

For sentences $\Gamma \cup \{\phi\} \subseteq \mathcal{L}$, we write $\Gamma \models \phi$ iff for every $\mathcal{L}$-model $\mathfrak{M}$ and every world $w \in \mathfrak{M}$: $w \Vdash \Gamma$ only if $w \Vdash \phi$. 

\begin{theorem}[Conditional Proof] If $\Gamma, \phi \models \psi$ then $\Gamma \models \phi \rightarrow \psi$.
\end{theorem}
\begin{proof} Suppose $\Gamma \not \models \phi \rightarrow \psi$.  Then there exist $w \in \mathfrak{M}$ such that $w \Vdash \Gamma$ and $w \not \Vdash \phi \rightarrow \psi$.  So $u \Vdash \phi$ and $u \not \Vdash \psi$ for some $u \succ w$. By Persistence, $u \Vdash \Gamma$.  But then $\Gamma, \phi \not \models \psi$.
\end{proof}

\begin{theorem}[Weak Modus Ponens] For $\Gamma \subseteq \mathcal{L} \setminus \{\rightarrow\}$: if $\Gamma \models \phi \rightarrow \psi$ then $\Gamma, \phi \models \psi$.
\end{theorem}
\begin{proof} Suppose $\Gamma, \phi \not \models \psi$. Then there exist $w \in \mathfrak{M}$ such that $w \Vdash \Gamma \cup \{\phi\}$ and $w \not \Vdash \psi$.  Take the submodel of $\mathfrak{M}$ generated by $w$. Add a new world $u$ below $w$ such that $|R^n|(u) = |R^n|(w)$. Then $u \Vdash \Gamma$ and $u \not \Vdash \phi \rightarrow \psi$.  So $\Gamma \not \models \phi \rightarrow \psi$.
\end{proof}

\begin{theorem}[Compactness] If $\Gamma \models \phi$ then $\Gamma_0 \models \phi$ for some finite $\Gamma_0 \subseteq \Gamma$.
\end{theorem}
\begin{proof} Similar to the ultraproduct proof of compactness for classical first-order logic (see e.g. Poizat [3]).
\end{proof}

\subsection{Proof Theory}

We formulate the natural deduction system $\mathcal{N}\textsf{BQL}_\textsf{CD}$ for $\textsf{BQL}_\textsf{CD}$ in the language $\mathcal{L}^+ = \mathcal{L} \cup \{a_i\}_{i \in \omega}$, where each $a_i$ is a fresh constant symbol (the $a_i$ serve in proofs as names of arbitrarily chosen objects).  $\mathcal{N}\textsf{BQL}_\textsf{CD}$ consists of all trees of (possibly discharged) $\mathcal{L}^+$-sentences constructed in accordance with the following inference rules:\begin{footnote}{By restricting the formulas linked by an inference rule to sentences we determine which free variables, if any, a subformula may contain (e.g. $\phi$ may not contain free variables in CD).}\end{footnote}
$$
[\top]\hspace{2mm}(\top\text{-Int})
\qquad
\infer[(\bot\text{-Elim})]
	{\phi}
	{\bot}
$$
$$
\infer[(\wedge\text{-Int})]
	{\phi \wedge \psi}
	{
	\phi
	&
	\psi
	}
\qquad
\infer[(\wedge\text{-Elim})]
	{\phi / \psi}
	{\phi \wedge \psi}
$$
$$
\infer[(\vee\text{-Int})]
	{\phi \vee \psi}
	{\phi / \psi}
\qquad
\infer[(\vee\text{-Elim})]
	{\chi}
	{
	\phi \vee \psi
	&
	\infer*
		{\chi}
		{[\phi]}
	&
	\infer*
		{\chi}
		{[\psi]}
	}
$$
$$
\infer[(\rightarrow\hspace{-1mm}\text{-Int})]
	{\phi \rightarrow \psi}
	{
	\infer*
		{\psi}
		{[\phi]}
	}
\qquad
\infer[(\text{Internal Transitivity})]
	{\phi \rightarrow \chi}
	{
	\phi \rightarrow \psi
	&
	\psi \rightarrow \chi
	}
$$
$$
\infer[(\text{Internal } \wedge\hspace{-1mm}\text{-Int})]
	{\phi \rightarrow \psi \wedge \chi}
	{
	\phi \rightarrow \psi
	&
	\phi \rightarrow \chi
	}
\qquad
\infer[(\text{Internal } \vee\hspace{-1mm}\text{-Elim})]
	{\phi \vee \psi \rightarrow \chi}
	{
	\phi \rightarrow \chi
	&
	\psi \rightarrow \chi
	}
$$
$$
\infer[(\text{Internal } \forall\text{-Int})]
	{\phi \rightarrow \forall v \psi}
	{\forall v (\phi \rightarrow \psi)}
\qquad
\infer[(\text{Internal } \exists\text{-Elim})]
	{\exists v \phi \rightarrow \psi}
	{\forall v (\phi \rightarrow \psi)}
$$
$$
\infer[(\forall\text{-Int})]
	{\forall v \phi}
	{\phi(a_i)}
\qquad
\infer[(\forall\text{-Elim})]
	{\phi(t)}
	{\forall v \phi}
\qquad
\infer[(\text{CD})]
	{\phi \vee \forall v \psi}
	{\forall v (\phi \vee \psi)}
$$
$$
\infer[(\exists\text{-Int})]
	{\exists v \phi}
	{\phi(t)}
\qquad
\infer[(\exists\text{-Elim})]
	{\psi}
	{
	\exists v \phi
	&
	\infer*
		{\psi}
		{[\phi(a_i)]}
	}
$$
In $\forall$-Int, $a_i$ does not occur in $\phi$ or in any open assumption in the main subproof.  In $\exists$-Elim, $a_i$ does not occur in $\phi$, $\psi$ or in any open assumption besides $\phi(a_i)$ in the right main subproof.  We write $\Gamma \vdash \phi$ iff there exists a proof of $\phi$ from $\Gamma$ in $\mathcal{N}\textsf{BQL}_\textsf{CD}$.  

\begin{theorem}[Soundness] If $\Gamma \vdash \phi$ then $\Gamma \models \phi$.
\end{theorem}
\begin{proof} By induction on the construction of proofs in $\mathcal{N}\textsf{BQL}_\textsf{CD}$.  The base case is easy.  The induction steps are also easy except for $\rightarrow$-Int, where we appeal to Conditional Proof.
\end{proof}

\begin{lemma}[Distribution] $\phi \wedge (\psi \vee \chi) \vdash (\phi \wedge \psi) \vee (\phi \wedge \chi)$.
\end{lemma}

\begin{lemma}[Infinite Distribution] $\phi \wedge \exists v \psi \vdash \exists v(\phi \wedge \psi)$.
\end{lemma}

For sentences $\Sigma \subseteq \mathcal{L}^+$, let $\mathcal{N}\textsf{BQL}_\textsf{CD}(\Sigma)$ denote the natural deduction system obtained by adding the rule
$$
\infer
	{\psi}
	{
	\phi
	&
	\infer
		{\phi \rightarrow \psi}
		{
		\infer
			{\mathcal{N}\textsf{BQL}_\textsf{CD}}
			{\Sigma}
		}
	}
$$
to $\mathcal{N}\textsf{BQL}_\textsf{CD}$ and adding a restriction which states that (i) no occurrence $\chi^i$ of an open assumption in position
$$
\infer*
	{}
	{
	...
	&
	\infer
		{\psi}
		{
		\infer*
			{\phi}
			{}
		&
		\infer*
			{\phi \rightarrow \psi}
			{\chi^i}
		}
	&
	...
	}
$$
can be discharged (we say that such occurrences are \textit{unsafe}) and (ii) no occurrence of $\phi(a_i)$ in the right main subproof of $\exists$-Elim is unsafe. Write $\Gamma \vdash_\Sigma \phi$ iff there exists a proof of $\phi$ from $\Gamma$ in $\mathcal{N}\textsf{BQL}_\textsf{CD}(\Sigma)$. To state the next lemma concisely, let $\bigwedge \emptyset = \top$. 

\begin{lemma}[Relative Deduction] For $|\Gamma| < \omega$, if there exists a proof $\Pi \in \mathcal{N}\textsf{BQL}_\textsf{CD}(\Sigma)$ of $\phi$ from $\Sigma' \cup \Gamma$ such that every open assumption which occurs unsafely in $\Pi$ is contained in $\Sigma'$ then $\Sigma' \vdash \bigwedge \Gamma \rightarrow \phi$.
\end{lemma}

\begin{proof} By induction on the construction of proofs in $\mathcal{N}\textsf{BQL}_\textsf{CD}(\Sigma)$.

\underline{Base Case} Suppose we have a one-line proof in $\mathcal{N}\textsf{BQL}_\textsf{CD}(\Sigma)$ of $\phi$ from $\Sigma' \cup \Gamma$.  There are three cases. 

\underline{Case 1} $\phi = \top$. Then
$$
\infer
	{\bigwedge \Gamma \rightarrow \top}
	{[\top]}
$$
is a proof of $\bigwedge \Gamma \rightarrow \phi$ from $\Sigma'$ in $\mathcal{N}\textsf{BQL}_\textsf{CD}$.  

\underline{Case 2} $\phi \in \Sigma'$. Then
$$
\infer
	{\bigwedge \Gamma \rightarrow \phi}
	{\phi}
$$
is a proof of $\bigwedge \Gamma \rightarrow \phi$ from $\Sigma'$ in $\mathcal{N}\textsf{BQL}_\textsf{CD}$.  

\underline{Case 3} $\phi \in \Gamma$. Then
$$
\infer
	{\bigwedge \Gamma \rightarrow \phi}
	{
	\infer
		{\phi}
		{
		\infer	
			{\wedge\text{-Elims}}
			{[\bigwedge \Gamma]}	
		}
	}
$$
is a proof of $\bigwedge \Gamma \rightarrow \phi$ from $\Sigma'$ in $\mathcal{N}\textsf{BQL}_\textsf{CD}$.

\underline{Induction Steps} There are seven cases.

\underline{Case 1} Suppose we have a proof of the form
$$
\infer
	{\phi}
	{
	\infer*
		{\alpha}
		{\Sigma', \Gamma}
	}
$$
in $\mathcal{N}\textsf{BQL}_\textsf{CD}(\Sigma)$, where the final inference is $\bot$-Elim, $\wedge$-Elim, $\vee$-Int, Internal $\forall$-Int, Internal $\exists$-Elim, $\forall$-Elim, CD or $\exists$-Int.  Then, by the induction hypothesis, we can find a proof of the form
$$
\infer
	{\bigwedge \Gamma \rightarrow \phi}
	{
	\infer*
		{\bigwedge \Gamma \rightarrow \alpha}
		{\Sigma'}
	&
	\infer
		{\alpha \rightarrow \phi}
		{
		\infer
			{\phi}
			{[\alpha]}
		}
	}
$$
in $\mathcal{N}\textsf{BQL}_\textsf{CD}$.

\underline{Case 2} Suppose we have a proof of the form
$$
\infer
	{\phi}
	{
	\infer*
		{\alpha}
		{\Sigma', \Gamma}
	&
	\infer*
		{\beta}
		{\Sigma', \Gamma}
	}
$$
in $\mathcal{N}\textsf{BQL}_\textsf{CD}(\Sigma)$, where the final inference is $\wedge$-Int, Internal Transitivity, Internal $\wedge$-Int or Internal $\vee$-Elim. Then, by the induction hypothesis, we can find a proof of the form
$$
\infer
	{\bigwedge \Gamma \rightarrow \phi}
	{
	\infer
		{\bigwedge \Gamma \rightarrow \alpha \wedge \beta}
		{
		\infer*
			{\bigwedge \Gamma \rightarrow \alpha}
			{\Sigma'}
		&
		\infer*
			{\bigwedge \Gamma \rightarrow \beta}
			{\Sigma'}
		}
	&
	\infer
		{\alpha \wedge \beta \rightarrow \phi}
		{
		\infer
			{\phi}
			{
			\infer
				{\alpha}
				{[\alpha \wedge \beta]}
			&
			\infer
				{\beta}
				{[\alpha \wedge \beta]}
			}
		}
	}
$$
in $\mathcal{N}\textsf{BQL}_\textsf{CD}$.

\underline{Case 3} Suppose we have a proof of the form
$$
\infer
	{\phi}
	{
	\infer*
		{\alpha \vee \beta}
		{\Sigma', \Gamma}		
	&
	\infer*
		{\phi}
		{\Sigma', \Gamma,  [\alpha]}		
	&
	\infer*
		{\phi}
		{\Sigma', \Gamma, [\beta]}	
	}
$$
in $\mathcal{N}\textsf{BQL}_\textsf{CD}(\Sigma)$. Since unsafe occurrences cannot be discharged, if $\alpha$ occurs unsafely in the center main subproof then $\alpha \in \Sigma'$, and the same goes for $\beta$ in the right main subproof. There are two subcases to consider.

\underline{Subcase 1} $\Gamma = \emptyset$. Then, by the induction hypothesis, we can find a proof of the form 
$$
\infer
	{\top \rightarrow \phi}
	{
	\infer*
		{\top \rightarrow \alpha \vee \beta}
		{\Sigma'}
	&
	\infer
		{\alpha \vee \beta \rightarrow \phi}
		{
		\infer*
			{\alpha \rightarrow \phi}
			{\Sigma'}
		&
		\infer*
			{\beta \rightarrow \phi}
			{\Sigma'}
		}
	}
$$
in $\mathcal{N}\textsf{BQL}_\textsf{CD}$.

\underline{Subcase 2} $\Gamma \neq \emptyset$. Then, by Distribution and the induction hypothesis, we can find a proof of the form
$$
\infer
	{\bigwedge\Gamma \wedge (\alpha \vee \beta) \rightarrow \phi}
	{
	\infer
		{\bigwedge\Gamma \wedge (\alpha \vee \beta) \rightarrow (\bigwedge \Gamma \wedge \alpha) \vee (\bigwedge \Gamma \wedge \beta)}
		{
		\infer*
			{(\bigwedge \Gamma \wedge \alpha) \vee (\bigwedge \Gamma \wedge \beta)}
			{[\bigwedge\Gamma \wedge (\alpha \vee \beta)]}
		}
	&
	\infer
		{(\bigwedge \Gamma \wedge \alpha) \vee (\bigwedge \Gamma \wedge \beta) \rightarrow \phi}
		{
		\infer*
			{(\bigwedge \Gamma \wedge \alpha) \rightarrow \phi}
			{\Sigma'}		
		&
		\infer*
			{(\bigwedge \Gamma \wedge \beta) \rightarrow \phi}
			{\Sigma'}
		}
	}
$$
in $\mathcal{N}\textsf{BQL}_\textsf{CD}$.  So, by the induction hypothesis, we can find a proof of the form 
$$
\infer
	{\bigwedge \Gamma \rightarrow \phi}
	{
	\infer
		{\bigwedge \Gamma \rightarrow \bigwedge \Gamma \wedge (\alpha \vee \beta)}
		{
		\infer
			{\bigwedge \Gamma \rightarrow \bigwedge \Gamma}
			{[\bigwedge \Gamma]}
		&
		\infer*
			{\bigwedge \Gamma \rightarrow \alpha \vee \beta}
			{\Sigma'}	
		}
	&
	\infer*
		{\bigwedge \Gamma \wedge (\alpha \vee \beta) \rightarrow \phi}
		{\Sigma'}
	}
$$
in $\mathcal{N}\textsf{BQL}_\textsf{CD}$. 

\underline{Case 4} Suppose we have a proof of the form
$$
\infer
	{\phi \rightarrow \psi}
	{
	\infer*
		{\psi}
		{\Sigma', \Gamma, [\phi]}
	}
$$
in $\mathcal{N}\textsf{BQL}_\textsf{CD}(\Sigma)$. Since unsafe occurrences cannot be discharged, if $\phi$ occurs unsafely in the main subproof then $\phi \in \Sigma'$. There are two subcases to consider.

\underline{Subcase 1} $\Gamma = \emptyset$. Then, by the induction hypothesis, we can find a proof of the form
$$
\infer
	{\top \rightarrow (\phi \rightarrow \psi)}
	{
	\infer*
		{\phi \rightarrow \psi}
		{\Sigma'}
	}
$$
in $\mathcal{N}\textsf{BQL}_\textsf{CD}$.

\underline{Subcase 2} $\Gamma \neq \emptyset$.  Then, by the induction hypothesis, we can find a proof of the form
$$
\infer
	{\bigwedge \Gamma \rightarrow (\phi \rightarrow \psi)}
	{
	\infer
		{\phi \rightarrow \psi}
		{
		\infer
			{\phi \rightarrow \bigwedge \Gamma \wedge \phi}
			{
			\infer
				{\bigwedge \Gamma \wedge \phi}
				{
				[\bigwedge \Gamma]
				&
				[\phi]
				}
			}
		&
		\infer*
			{\bigwedge \Gamma \wedge \phi \rightarrow \psi}
			{\Sigma'}
		}
	}
$$
in $\mathcal{N}\textsf{BQL}_\textsf{CD}$.

\underline{Case 5} Suppose we have a proof of the form
$$
\infer
	{\phi}
	{
	\infer*
		{\alpha}
		{\Sigma', \Gamma}
	&
	\infer
		{\alpha \rightarrow \phi}
		{
		\infer
			{\mathcal{N}\textsf{BQL}_\textsf{CD}}
			{\Sigma'}
		}
	}
$$
in $\mathcal{N}\textsf{BQL}_\textsf{CD}(\Sigma)$. Then, by the induction hypothesis, we can find a proof of the form
$$
\infer
	{\bigwedge \Gamma \rightarrow \phi}
	{
	\infer*
		{\bigwedge \Gamma \rightarrow \alpha}
		{\Sigma'}
	&
	\infer
		{\alpha \rightarrow \phi}
		{
		\infer
			{\mathcal{N}\textsf{BQL}_\textsf{CD}}
			{\Sigma'}
		}
	}
$$
in $\mathcal{N}\textsf{BQL}_\textsf{CD}$.

\underline{Case 6} Suppose we have a proof of the form
$$
\infer
	{\forall v \phi}
	{
	\infer*
		{\phi(a_i)}
		{\Sigma', \Gamma}
	}
$$
in $\mathcal{N}\textsf{BQL}_\textsf{CD}(\Sigma)$.  Let $\Sigma^* \subseteq \Sigma', \Gamma^* \subseteq \Gamma$ contain exactly the open assumptions in the main subproof.  Then $a_i$ does not occur in $\Sigma^* \cup \Gamma^* \cup \{\phi\}$. So, by the induction hypothesis, we can find a proof of the form
$$
\infer
	{\bigwedge \Gamma \rightarrow \forall v \phi}
	{
	\infer
		{\bigwedge \Gamma \rightarrow \bigwedge \Gamma^*}
		{
		\infer
			{\bigwedge \Gamma^*}
			{
			\infer
				{\wedge\text{-Ints}}
				{
				\infer
					{\wedge\text{-Elims}}
					{[\bigwedge \Gamma]}
				}
			}
		}
	&
	\infer
		{\bigwedge \Gamma^* \rightarrow \forall v \phi}
		{
		\infer
			{\forall v(\bigwedge \Gamma^* \rightarrow \phi)}
			{
			\infer*
				{\bigwedge \Gamma^* \rightarrow \phi(a_i)}
				{\Sigma^*}
			}
		}
	}
$$
in $\mathcal{N}\textsf{BQL}_\textsf{CD}$.

\underline{Case 7} Suppose we have a proof of the form
$$
\infer
	{\phi}
	{
	\infer*
		{\exists v \psi}
		{\Sigma', \Gamma}
	&
	\infer*
		{\phi}
		{\Sigma', \Gamma, [\psi(a_i)]}
	}
$$
in $\mathcal{N}\textsf{BQL}_\textsf{CD}(\Sigma)$.  Let $\Sigma^* \subseteq \Sigma', \Gamma^* \subseteq \Gamma$ contain exactly the open assumptions other than $\psi(a_i)$ in the right main subproof.  Then $a_i$ does not occur in $\Sigma^* \cup \Gamma^* \cup \{\psi, \phi\}$. Furthermore, since $\psi(a_i)$ does not occur unsafely in the right main subproof, $\Sigma^*$ contains all open assumptions which occur unsafely in the right main subproof.

\underline{Subcase 1} $\Gamma^* = \emptyset$.  Then, by the induction hypothesis, we can find a proof of the form
$$
\infer
	{\bigwedge \Gamma \rightarrow \phi}
	{
	\infer*
		{\bigwedge \Gamma \rightarrow \exists v \psi}
		{\Sigma'}
	&
	\infer
		{\exists v \psi \rightarrow \phi}
		{
		\infer
			{\forall v (\psi \rightarrow \phi)}
			{
			\infer*
				{\psi(a_i) \rightarrow \phi}
				{\Sigma^*}
			}
		}
	}
$$
in $\mathcal{N}\textsf{BQL}_\textsf{CD}$.

\underline{Subcase 2} $\Gamma^* \neq \emptyset$. Then, by Infinite Distribution and the induction hypothesis, we can find a proof of the form
$$
\infer
	{\bigwedge \Gamma \wedge \exists v \psi \rightarrow \phi}
	{
	\infer
		{\bigwedge \Gamma \wedge \exists v \psi \rightarrow \bigwedge \Gamma^* \wedge \exists v \psi}
		{
		\infer
			{\bigwedge \Gamma^* \wedge \exists v \psi}
			{
			\infer
				{\wedge\text{-Ints}}
				{
				\infer
					{\wedge\text{-Elims}}
					{[\bigwedge \Gamma \wedge \exists v \psi]}
				}
			}
		}
	&
	\infer
		{\bigwedge \Gamma^* \wedge \exists v \psi \rightarrow \phi}
		{
		\infer
			{\bigwedge \Gamma^* \wedge \exists v \psi \rightarrow \exists v(\bigwedge \Gamma^* \wedge \psi)}
			{
			\infer*
				{\exists v(\bigwedge \Gamma^* \wedge \psi)}
				{[\bigwedge \Gamma^* \wedge \exists v \psi]}
			}
		&
		\infer
			{\exists v(\bigwedge \Gamma^* \wedge \psi) \rightarrow \phi}
			{
			\infer
				{\forall v (\bigwedge \Gamma^* \wedge \psi \rightarrow \phi)}
				{
				\infer*
					{\bigwedge \Gamma^* \wedge \psi(a_i) \rightarrow \phi}
					{\Sigma^*}
				}
			}
		}
	}
$$
in $\mathcal{N}\textsf{BQL}_\textsf{CD}$.  So, by the induction hypothesis, we can find a proof of the form
$$
\infer
	{\bigwedge \Gamma \rightarrow \phi}
	{
	\infer
		{\bigwedge \Gamma \rightarrow \bigwedge \Gamma \wedge \exists v \psi}
		{
		\infer
			{\bigwedge \Gamma \rightarrow \bigwedge \Gamma}
			{[\bigwedge \Gamma]}
		&
		\infer*
			{\bigwedge \Gamma \rightarrow \exists v \psi}
			{\Sigma'}
		}
	&
	\infer*
		{\bigwedge \Gamma \wedge \exists v \psi \rightarrow \phi}
		{\Sigma^*}
	}
$$
in $\mathcal{N}\textsf{BQL}_\textsf{CD}$. 
\end{proof}

\begin{corollary} For $|\Gamma| < \omega$, if $\Sigma, \Gamma \vdash_\Sigma \phi$ then $\Sigma \vdash \bigwedge \Gamma \rightarrow \phi$.
\end{corollary}

\section{The Canonical Model}

In order to prove the existence of the canonical model, we need to assume that $\mathcal{L}$ is countable (see Relative Extension below for an explanation). Fortunately, we can use compactness to leverage up our canonical model proofs of completeness, the disjunction property and the existence property to $\mathcal{L}$ of arbitrary cardinality. 

A set of sentences $\Gamma \subseteq \mathcal{L}^+$ is called a prime saturated $\textsf{BQL}_\textsf{CD}$-theory iff $\Gamma$ satisfies the following properties:
\begin{align*}
&(\text{consistency}) &&\bot \not \in \Gamma \\
&(\textsf{BQL}_\textsf{CD}\text{-closure}) &&\text{if } \Gamma \vdash \phi \text{ then } \phi \in \Gamma \\
&(\text{disjunction property}) &&\text{if } \phi \vee \psi \in \Gamma \text{ then } \phi \in \Gamma \text{ or } \psi \in \Gamma \\
&(\text{existence property}) &&\text{if } \exists v \phi \in \Gamma \text{ then } \phi(t) \in \Gamma \text{ for some } t \in \mathcal{L}^+ \\
&(\text{totality property}) &&\text{if } \phi(t) \in \Gamma \text{ for every } t \in \mathcal{L}^+ \text{ then } \forall v \phi \in \Gamma.
\end{align*}
Let $Sat(\textsf{BQL}_\textsf{CD})$ denote the set of prime saturated $\textsf{BQL}_\textsf{CD}$-theories.  

\begin{lemma}[Extension] For $\Gamma$ such that $|\{i: a_i \not \in \Gamma\}| = \omega$: if $\Gamma \not \vdash \phi$ then there exists $\Gamma^* \supseteq \Gamma$ such that $\Gamma^* \in Sat(\textsf{BQL}_\textsf{CD})$ and $\phi \not \in \Gamma^*$.
\end{lemma}
\begin{proof} Similar to the proof of the Belnap Extension Lemma (see e.g. Priest [4] \S 6.2), except we draw witnesses from $\{a_i\}_{i \in \omega}$. Since $\mathcal{L}$ is countable, the assumption that $|\{i: a_i \not \in \Gamma\}| = \omega$ ensures we never run out of witnesses.
\end{proof}

\begin{lemma}[$\exists$-Witness] For $\Sigma \in Sat(\textsf{BQL}_\textsf{CD})$, $|\Gamma| < \omega$: if $\Sigma, \Gamma, \psi(t) \vdash_\Sigma \phi$ for every $t \in \mathcal{L}^+$ then $\Sigma, \Gamma, \exists v \psi \vdash_\Sigma \phi$.
\end{lemma}
\begin{proof} Suppose $\Sigma, \Gamma, \psi(t) \vdash_\Sigma \phi$ for every $t \in \mathcal{L}^+$.  There are two cases.

\underline{Case 1} $\Gamma = \emptyset$. Then, by Relative Deduction, $\Sigma \vdash \psi(t) \rightarrow \phi$ for every $t \in \mathcal{L}^+$. Since $\Sigma \in Sat(\textsf{BQL}_\textsf{CD})$, $\Sigma \vdash \forall v(\psi \rightarrow \phi)$. So $\Sigma \vdash \exists v \psi \rightarrow \phi$.  But then $\Sigma, \exists v \psi \vdash_\Sigma \phi$.

\underline{Case 2} $\Gamma \neq \emptyset$. Then, by Relative Deduction, $\Sigma \vdash \bigwedge \Gamma \wedge \psi(t) \rightarrow \phi$ for every $t \in \mathcal{L}^+$. Since $\Sigma \in Sat(\textsf{BQL}_\textsf{CD})$, $\Sigma \vdash \forall v(\bigwedge \Gamma \wedge \psi \rightarrow \phi)$. So $\Sigma \vdash \exists v(\bigwedge \Gamma \wedge \psi) \rightarrow \phi$. But then, by Infinite Distributivity, $\Sigma \vdash \bigwedge \Gamma \wedge \exists v \psi \rightarrow \phi$.  So $\Sigma, \Gamma, \exists v \psi \vdash_\Sigma \phi$.
\end{proof}

\begin{lemma}[$\forall$-Witness] For $\Sigma \in Sat(\textsf{BQL}_\textsf{CD})$, $|\Gamma| < \omega$: if $\Sigma, \Gamma \vdash_\Sigma \phi \vee \psi(t)$ for every $t \in \mathcal{L}^+$ then $\Sigma, \Gamma \vdash_\Sigma \phi \vee \forall v \psi$.
\end{lemma}
\begin{proof} Suppose $\Sigma, \Gamma \vdash_\Sigma \phi \vee \psi(t)$ for every $t \in \mathcal{L}^+$.  Then, by Relative Deduction, $\Sigma \vdash \bigwedge \Gamma \rightarrow \phi \vee \psi(t)$ for every $t \in \mathcal{L}^+$. Since $\Sigma \in Sat(\textsf{BQL}_\textsf{CD})$, $\Sigma \vdash \forall v(\bigwedge \Gamma \rightarrow \phi \vee \psi)$.  So $\Sigma \vdash \bigwedge \Gamma \rightarrow \forall v(\phi \vee \psi)$. But then, by CD, $\Sigma \vdash \bigwedge \Gamma \rightarrow \phi \vee \forall v \psi$.  Hence $\Sigma, \Gamma \vdash_\Sigma \phi \vee \forall v \psi$.
\end{proof}

We define $Sat(\textsf{BQL}_\textsf{CD}(\Sigma))$ analogously to $Sat(\textsf{BQL}_\textsf{CD})$.

\begin{lemma}[Relative Extension] For $\Sigma \in Sat(\textsf{BQL}_\textsf{CD})$, $|\Gamma| < \omega$: if $\Sigma, \Gamma \not \vdash_\Sigma \phi$ then there exists $(\Sigma \cup \Gamma)^* \supseteq \Sigma \cup \Gamma$ such that $(\Sigma \cup \Gamma)^* \in Sat(\textsf{BQL}_\textsf{CD}(\Sigma))$ and $\phi \not \in (\Sigma \cup \Gamma)^*$.
\end{lemma}
\begin{proof} Suppose $\Sigma, \Gamma \not \vdash_\Sigma \phi$.  Since $\mathcal{L}$ is countable, we can fix an enumeration $\{\phi_n\}_{n \in \omega}$ of $\mathcal{L}^+$-sentences.  We then inductively define a pair $\{\Gamma_n\}_{n \in \omega}, \{\Delta_n\}_{n \in \omega}$ of increasing sequences of sets of sentences $\Gamma_n, \Delta_n \subseteq \mathcal{L}^+$ as follows, where $\Pi^L_n(\psi) = \{t \in \mathcal{L}^+: \Sigma, \Gamma_n, \psi(t) \not \vdash_\Sigma \bigvee \Delta_n\}$ and $\Pi^R_n(\psi) = \{t \in \mathcal{L}^+: \Sigma, \Gamma_n \not \vdash_\Sigma \bigvee \Delta_n \vee \psi(t)\}$:
\begin{align*}
\Gamma_0 &= \Gamma \\
\Gamma_{n + 1} &= \begin{cases}
					\Gamma_n \text{ if } \Sigma, \Gamma_n, \phi_n \vdash_\Sigma \bigvee \Delta_n \\
					\Gamma_n \cup \{\phi_n\} \text{ if } \Sigma, \Gamma_n, \phi_n \not \vdash_\Sigma \bigvee \Delta_n \text{ and } \phi_n \neq \exists v \psi \\
					\Gamma_n \cup \{\phi_n, \psi(t)\} \text{ for } t \in \Pi^L_n(\psi) \text{ if } \Sigma, \Gamma_n, \phi_n \not \vdash_\Sigma \bigvee \Delta_n \text{ and } \phi_n = \exists v \psi
					\end{cases}		\\
\Delta_0 &= \{\phi\} \\
\Delta_{n + 1} &= \begin{cases}
					\Delta_n \text{ if } \Sigma, \Gamma_n, \phi_n \not \vdash_\Sigma \bigvee \Delta_n \\
					\Delta_n \cup \{\phi_n\} \text{ if } \Sigma, \Gamma_n, \phi_n \vdash_\Sigma \bigvee \Delta_n \text{ and } \phi_n \neq \forall v \psi \\
					\Delta_n \cup \{\phi_n, \psi(t)\} \text{ for } t \in \Pi^R_n(\psi) \text{ if } \Sigma, \Gamma_n, \phi_n \vdash_\Sigma \bigvee \Delta_n \text{ and } \phi_n = \forall v \psi.
					\end{cases}
\end{align*}
In order for this construction to be well-defined, we require $\Pi^L_n(\psi) \neq \emptyset$ at the stages where we choose $t \in \Pi^L_n(\psi)$ (likewise for $\Pi^R_n(\psi)$).  We prove this by appealing to $\exists$-Witness ($\forall$-Witness), which requires $|\Gamma_n| < \omega$.  Hence $\mathcal{L}$ must be countable, for otherwise we would need to iterate the above construction into the transfinite.

\begin{subclaim}[Separation] For all $n$: (i) $\Gamma_n, \Delta_n$ exist, (ii) $|\Gamma_n|, |\Delta_n| < \omega$ and (iii) $\Sigma, \Gamma_n \not \vdash_\Sigma \bigvee \Delta_n$.
\end{subclaim}
\begin{subproof} By induction on $n$.  The base case is immediate.  

\underline{Induction Step} There are two cases.

\underline{Case 1} $\Sigma, \Gamma_n, \phi_n \not \vdash_\Sigma \bigvee \Delta_n$.  Then $\Delta_{n+1} = \Delta_n$.  So if $\phi_n \neq \exists v \psi$ then $\Gamma_{n + 1} = \Gamma_n \cup \{\phi_n\}$ and we're done.  Suppose, then, that $\phi_n = \exists v \psi$.  Then, by $\exists$-Witness, $\Pi^L_n(\psi) \neq \emptyset$ and so $\Gamma_{n + 1} = \Gamma_n \cup \{\exists v \psi, \psi(t)\}$ for some $t \in \Pi^L_n(\psi)$ exists.  Suppose for a reductio that $\Sigma, \Gamma_{n + 1} \vdash_\Sigma \bigvee \Delta_{n + 1}$.  Then, by Relative Deduction, $\Sigma \vdash \bigwedge \Gamma_{n + 1} \rightarrow \bigvee \Delta_{n + 1}$. There are two subcases.

\underline{Subcase 1} $\Gamma_n = \emptyset$. Then we can construct the following proof in $\mathcal{N}\textsf{BQL}_\textsf{CD}(\Sigma)$:
$$
\infer
	{\bigvee \Delta_n}
	{
	\infer
		{\bigwedge \Gamma_{n+1}}
		{
		\psi(t)
		&
		\infer
			{\exists v \psi}
			{\psi(t)}
		}
	&
	\infer
		{\bigwedge \Gamma_{n+1} \rightarrow \bigvee \Delta_n}
		{
		\infer
			{\mathcal{N}\textsf{BQL}_\textsf{CD}}
			{\Sigma}
		}
	}
$$
which contradicts the fact that $t \in \Pi^L_n(\psi)$.

\underline{Subcase 2} $\Gamma_n \neq \emptyset$.  Then we can construct the following proof in $\mathcal{N}\textsf{BQL}_\textsf{CD}(\Sigma)$:
$$
\infer
	{\bigvee \Delta_n}
	{
	\infer
		{\bigwedge \Gamma_n \wedge \psi(t)}
		{
		\infer
			{\wedge\text{-Ints}}
			{\Gamma_n, \psi(t)}
		}
	&
	\infer
		{\bigwedge \Gamma_n \wedge \psi(t) \rightarrow \bigvee \Delta_n}
		{
		\infer
			{\bigwedge \Gamma_n \wedge \psi(t) \rightarrow \bigwedge \Gamma_{n + 1}}
			{
			\infer
				{\bigwedge \Gamma_{n + 1}}
				{
				[\bigwedge \Gamma_n \wedge \psi(t)]
				&
				\infer
					{\exists v \psi}
					{
					\infer
						{\psi(t)}
						{[\bigwedge \Gamma_n \wedge \psi(t)]}
					}
				}
			}
		&
		\infer
			{\bigwedge \Gamma_{n + 1} \rightarrow \bigvee \Delta_n}
			{
			\infer
				{\mathcal{N}\textsf{BQL}_\textsf{CD}}
				{\Sigma}
			}
		}
	}
$$
which contradicts the fact that $t \in \Pi^L_n(\psi)$.

\underline{Case 2} $\Sigma, \Gamma_n, \phi_n \vdash_\Sigma \bigvee \Delta_n$.  Then $\Gamma_{n+1} = \Gamma_n$.  Also, by the induction hypothesis, $\phi_n \not \in \Sigma$. There are two subcases.

\underline{Subcase 1}  $\phi_n \neq \forall v \psi$. Then $\Delta_{n + 1} = \Delta_n \cup \{\phi_n\}$.  Suppose for a reductio that $\Sigma, \Gamma_{n + 1} \vdash_\Sigma \bigvee \Delta_{n + 1}$. Since $\phi_n \not \in \Sigma$, $\phi_n$ never occurs unsafely in a proof in $\mathcal{N}\textsf{BQL}_\textsf{CD}(\Sigma)$. So we can construct the following proof in $\mathcal{N}\textsf{BQL}_\textsf{CD}(\Sigma)$:
$$
\infer
	{\bigvee \Delta_n}
	{
	\infer*
		{\bigvee \Delta_n \vee \phi_n}
		{\Sigma, \Gamma_n}
	&
	[\bigvee \Delta_n]
	&
	\infer*
		{\bigvee \Delta_n}
		{\Sigma, \Gamma_n, [\phi_n]}
	}
$$
which contradicts the induction hypothesis.

\underline{Subcase 2} $\phi_n = \forall v \psi$.  By the same argument as Subcase 1: $\Sigma, \Gamma_n \not \vdash_\Sigma \bigvee \Delta_n \vee \forall v \psi$.  So, by $\forall$-Witness, $\Pi^R_n(\psi) \neq \emptyset$ and hence $\Delta_{n + 1} = \Delta_n \cup \{\forall v \psi, \psi(t)\}$ for some $t \in \Pi^R_n(\psi)$ exists.  Suppose for a reductio that $\Sigma, \Gamma_{n + 1} \vdash_\Sigma \bigvee \Delta_{n + 1}$.  Then we can construct the following proof in $\mathcal{N}\textsf{BQL}_\textsf{CD}(\Sigma)$:
$$
\infer
	{\bigvee \Delta_n \vee \psi(t)}
	{
	\infer*
		{(\bigvee \Delta_n \vee \psi(t)) \vee \forall v \psi}
		{\Sigma, \Gamma_n}
	&
	[\bigvee \Delta_n \vee \psi(t)]
	&
	\infer
		{\bigvee \Delta_n \vee \psi(t)}
		{
		\infer
			{\psi(t)}
			{[\forall v \psi]}
		}
	}
$$
which contradicts the fact that $t \in \Pi^R_n(\psi)$.  
\end{subproof}

Let $(\Sigma \cup \Gamma)^* = \Sigma \cup \bigcup_{n \in \omega} \Gamma_n$.  It is easy to verify using Separation that $(\Sigma \cup \Gamma)^*$ is the desired prime saturated $\textsf{BQL}_\textsf{CD}(\Sigma)$-theory.
\end{proof}

The canonical frame (for $\textsf{BQL}_\textsf{CD}$) is $\langle Sat(\textsf{BQL}_\textsf{CD}), \prec \rangle$, where $\Sigma \prec \Gamma$ iff for all $\phi, \psi$: if $\phi \rightarrow \psi \in \Sigma$ and $\phi \in \Gamma$ then $\psi \in \Gamma$.  By soundness, $R^n(a_1,...,a_n) \not \vdash \bot$.  So, by Extension, $Sat(\textsf{BQL}_\textsf{CD})$ is non-empty. 

\begin{lemma}[Subset] If $\Sigma \prec \Gamma$ then $\Sigma \subseteq \Gamma$.
\end{lemma}
\begin{proof} Suppose $\Sigma \prec \Gamma$ and $\phi \in \Sigma$.  Then $\top \rightarrow \phi \in \Sigma$.  So, since $\top \in \Gamma$, $\phi \in \Gamma$.
\end{proof}

To verify transitivity, suppose $\Sigma \prec \Gamma \prec \Delta$, $\phi \rightarrow \psi \in \Sigma$ and $\phi \in \Delta$.  Then, by Subset, $\phi \rightarrow \psi \in \Gamma$.  So $\psi \in \Delta$.  Hence the canonical frame is in fact a transitive frame. The canonical model (for $\textsf{BQL}_\textsf{CD}$) is $\mathfrak{C} = \langle Sat(\textsf{BQL}_\textsf{CD}), \prec, T, |\mathord{\cdot}| \rangle$, where $T$ is the set of closed $\mathcal{L}^+$-terms and
\begin{align*}
|c| &= c \\
|f^n|(t_1,...,t_n) &= f^n(t_1,...,t_n) \\
|R^n|(\Sigma) &= \{\langle t_1,...,t_n \rangle: R^n(t_1,...,t_n) \in \Sigma\}.
\end{align*}
By Subset, we have
\begin{align*}
\Sigma \prec \Gamma &\implies \Sigma \subseteq \Gamma \\
&\implies |R^n|(\Sigma) \subseteq |R^n|(\Gamma).
\end{align*}
So $\mathfrak{C}$ is in fact an $\mathcal{L}^+$-model.  

\begin{lemma}[Truth] $\mathfrak{C}, \Sigma \Vdash \phi$ iff $\phi \in \Sigma$.
\end{lemma}
\begin{proof} By induction on the complexity of $\mathcal{L}^+$-sentences.  The base case is easy.  The induction steps are also easy except for $\rightarrow$. 

\underline{$\rightarrow$} $\impliedby$ Easy.

$\implies$ Suppose $\phi \rightarrow \psi \not \in \Sigma$.  Then $\Sigma \not \vdash \phi \rightarrow \psi$.  By Relative Deduction: $\Sigma, \phi \not \vdash_\Sigma \psi$.  Hence, by Relative Extension, there exists $(\Sigma \cup \{\phi\})^* \supseteq \Sigma \cup \{\phi\}$ such that $(\Sigma \cup \{\phi\})^* \in Sat(\textsf{BQL}_\textsf{CD}(\Sigma))$ and $\psi \not \in (\Sigma \cup \{\phi\})^*$.  So $(\Sigma \cup \{\phi\})^* \in Sat(\textsf{BQL}_\textsf{CD})$.  Then, by the induction hypothesis, $(\Sigma \cup \{\phi\})^* \Vdash \phi$ and $(\Sigma \cup \{\phi\})^* \not \Vdash \psi$.  But $\Sigma \prec (\Sigma \cup \{\phi\})^*$.  So $\Sigma \not \Vdash \phi \rightarrow \psi$.
\end{proof}

\section{Completeness}

We now drop that assumption that $\mathcal{L}$ is countable. We prove that completeness holds over the extended language $\mathcal{L}^+$.

\begin{lemma}[Weak Completeness] For $|\Gamma| < \omega$: if $\Gamma \models \phi$ then $\Gamma \vdash \phi$. 
\end{lemma}
\begin{proof} Suppose $\Gamma \not \vdash \phi$.  Since $|\Gamma| < \omega$, we can find a countable first-order language $\mathcal{L}_0 \subseteq \mathcal{L}$ such that $\Gamma \cup \{\phi\} \subseteq \mathcal{L}_0^+$.  A forteriori, there does not exist a proof in $\mathcal{N}\textsf{BQL}_\textsf{CD} \upharpoonright \mathcal{L}_0^+$ of $\phi$ from $\Gamma$. Since $\{i: a_i \not \in \Gamma\} = \omega$, Extension gives $\Gamma^* \supseteq \Gamma$ such that $\Gamma^* \in Sat(\textsf{BQL}_\textsf{CD})$ (where $Sat(\textsf{BQL}_\textsf{CD})$ is defined over $\mathcal{L}_0^+$) and $\phi \not \in \Gamma^*$.  Let $\mathfrak{C}$ be the canonical model over $\mathcal{L}_0^+$. Then, by Truth: $\mathfrak{C}, \Gamma^* \Vdash \Gamma$ and $\mathfrak{C}, \Gamma^* \not \Vdash \phi$.  So an arbitrary expansion of $\mathfrak{C}$ to $\mathcal{L}^+$ gives $\Gamma \not \models \phi$. 
\end{proof}

\begin{theorem}[Completeness] If $\Gamma \models \phi$ then $\Gamma \vdash \phi$.
\end{theorem}
\begin{proof} Immediate from compactness and weak completeness.
\end{proof}

\section{Disjunction and Existence Properties}

The canonical model is useful for finding families of models which share a domain and agree on the interpretations of all constant symbols and function symbols.  This allows us to prove that $\textsf{BQL}_\textsf{CD}$ (over the original language $\mathcal{L}$) satisfies the disjunction and existence properties.

\begin{lemma}[Intersection] For $I \neq \emptyset$, let $\{w\} \cup \{u_i\}_{i \in I} \subseteq \mathfrak{M}$ be such that $|R^n|(w) = \bigcap_{i \in I} |R^n|(u_i)$.  Then, for $\phi(\overline{v}) \in \mathcal{L} \setminus \{\rightarrow, \vee, \exists\}$: $w \Vdash \phi(\overline{a})$ iff for all $i$: $u_i \Vdash \phi(\overline{a})$.
\end{lemma}
\begin{proof} An easy induction on the construction of $\mathcal{L} \setminus \{\rightarrow, \vee, \exists\}$-formulas.
\end{proof}

\begin{lemma}[Weak Disjunction Property] For $\Gamma \subseteq \mathcal{L} \setminus \{\rightarrow, \vee, \exists\}$ such that $|\Gamma| \leq \omega$: if $\Gamma \models \phi \vee \psi$ then $\Gamma \models \phi$ or $\Gamma \models \psi$.
\end{lemma}
\begin{proof} Suppose $\Gamma \not \models \phi$ and $\Gamma \not \models \psi$.  Then, by soundness, $\Gamma \not \vdash \phi$ and $\Gamma \not \vdash \psi$.  Since $|\Gamma| \leq \omega$, we can find a countable first-order language $\mathcal{L}_0 \subseteq \mathcal{L}$ such that $\Gamma \cup \{\phi, \psi\} \subseteq \mathcal{L}_0$. A forteriori, neither $\phi$ nor $\psi$ is provable from $\Gamma$ in $\mathcal{N}\textsf{BQL}_\textsf{CD} \upharpoonright \mathcal{L}_0^+$.  Since $\{i: a_i \in \Gamma\} = \emptyset$, Extension gives $\Gamma_\phi, \Gamma_\psi \in Sat(\textsf{BQL}_\textsf{CD})$ (where $Sat(\textsf{BQL}_\textsf{CD})$ is defined over $\mathcal{L}_0^+$) such that (i) $\Gamma \subseteq \Gamma_\phi$ and $\phi \not \in \Gamma_\phi$ and (ii) $\Gamma \subseteq \Gamma_\psi$ and $\psi \not \in \Gamma_\psi$.  Let $\mathfrak{C}$ be the canonical model over $\mathcal{L}_0^+$. Then, by Truth, we have (i) $\mathfrak{C}, \Gamma_\phi \Vdash \Gamma$ and $\mathfrak{C}, \Gamma_\phi \not \Vdash \phi$ and (ii) $\mathfrak{C}, \Gamma_\psi \Vdash \Gamma$ and $\mathfrak{C}, \Gamma_\psi \not \Vdash \psi$.  Let $f(\mathfrak{C})$ be a worlds-disjoint copy of $\mathfrak{C}$ obtained by replacing every $\Sigma \in Sat(\textsf{BQL}_\textsf{CD})$ with $f(\Sigma)$ and leaving everything else unchanged. Let $\mathfrak{C}_\phi$ denote the submodel of $\mathfrak{C}$ generated by $\Gamma_\phi$ and $f(\mathfrak{C})_\psi$ denote the submodel of $f(\mathfrak{C})$ generated by $f(\Gamma_\psi)$. Consider the following $\mathcal{L}^+_0$-model:
$$
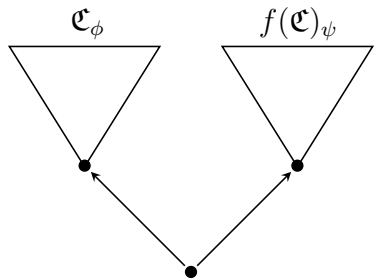
\begin{tikzpicture}[modal,node distance=2cm,world/.append style={minimum
size=1cm}]
\node[point] (root) {};

\node[point] (left) [above left of=root] {};

\node[point] (right) [above right of=root] {};

\path[->] (root) edge (left);

\path[->] (root) edge (right);

\draw (-1.415, 1.4)
  -- (-2.415,3) 
  -- (-0.415,3) 
  -- cycle;
  
\draw (1.415, 1.4)
  -- (0.415,3) 
  -- (2.415,3) 
  -- cycle; 
  
\draw (-1.5, 3.3) node[text width = 0.2cm] 
         {$\mathfrak{C}_\phi$}; 
         
\draw (1, 3.3) node[text width = 0.2cm] 
         {$f(\mathfrak{C})_\psi$}; 
\end{tikzpicture}
$$
where the root $w$ is a new world such that $|R^n|(w) = |R^n|(\Gamma_\phi) \cap |R^n|(f(\Gamma_\psi))$.  By Intersection, $w \Vdash \Gamma$.  By Persistence, $w \not \Vdash \phi$ and $w \not \Vdash \psi$.  So $w \not \Vdash \phi \vee \psi$.  Taking the reduct of this model to $\mathcal{L}_0$ and then arbitrarily expanding to $\mathcal{L}$ gives $\Gamma \not \models \phi \vee \psi$. 
\end{proof}

\begin{theorem}[Disjunction Property] For $\Gamma \subseteq \mathcal{L} \setminus \{\rightarrow, \vee, \exists\}$: if $\Gamma \models \phi \vee \psi$ then $\Gamma \models \phi$ or $\Gamma \models \psi$.
\end{theorem}
\begin{proof} Immediate from compactness and the weak disjunction property.
\end{proof}

\begin{lemma}[Weak Existence Property] Suppose $\mathcal{L}$ contains at least one constant symbol. Then, for $\Gamma \subseteq \mathcal{L} \setminus \{\rightarrow, \vee, \exists\}$ such that $|\Gamma| \leq \omega$: $\Gamma \models \exists v \phi$ only if $\Gamma \models \phi(t)$ for some $t \in \mathcal{L}$.
\end{lemma}
\begin{proof} Suppose $\Gamma \not \models \phi(t)$ for every $t \in \mathcal{L}$.  Suppose for a reductio that $\Gamma \models \phi(t)$ for some $t \in \mathcal{L}^+$.  Then, since $t = t_0(\overline{a_i})$ for some $t_0(\overline{u}) \in \mathcal{L}$, we have by $\forall$-Int that $\Gamma \models \forall \overline{u} \phi(t_0)$. So $\Gamma \models \phi(t_0(c,...,c))$ for some $c \in \mathcal{L}$, which is a contradiction. Therefore $\Gamma \not \models \phi(t)$ for every $t \in \mathcal{L}^+$. So, by soundness, $\Gamma \not \vdash \phi(t)$ for every $t \in \mathcal{L}^+$. Since $|\Gamma| \leq \omega$, we can find a countable first-order language $\mathcal{L}_0 \subseteq \mathcal{L}$ such that $\Gamma \cup \{\phi\} \subseteq \mathcal{L}_0$.  A forteriori, for all $t \in \mathcal{L}_0^+$ there does not exist a proof of $\phi(t)$ from $\Gamma$ in $\mathcal{N}\textsf{BQL}_\textsf{CD} \upharpoonright \mathcal{L}_0^+$. Since $\{i: a_i \in \Gamma\} = \emptyset$, Extension gives us a family $\{\Gamma_t\}_{t \in \mathcal{L}_0^+} \subseteq Sat(\textsf{BQL}_\textsf{CD})$ (where $Sat(\textsf{BQL}_\textsf{CD})$ is defined over $\mathcal{L}_0^+$) such that $\Gamma \subseteq \Gamma_t$ and $\phi(t) \not \in \Gamma_t$. Let $\mathfrak{C}$ be the canonical model over $\mathcal{L}_0^+$.  Then, by Truth: $\mathfrak{C}, \Gamma_t \Vdash \Gamma$ and $\mathfrak{C}, \Gamma_t \not \Vdash \phi(t)$.  Let $\{f_t(\mathfrak{C})\}_{t \in \mathcal{L}_0^+}$ be a family of pairwise worlds-disjoint copies of $\mathfrak{C}$ such that $f_t(\mathfrak{C})$ is obtained by replacing every $\Sigma \in Sat(\textsf{BQL}_\textsf{CD})$ with $f_t(\Sigma)$ and leaving everything else unchanged. Let $f_t(\mathfrak{C})^*$ denote the submodel of $f_t(\mathfrak{C})$ generated by $f_t(\Gamma_t)$. Consider the following $\mathcal{L}_0^+$-model:
$$
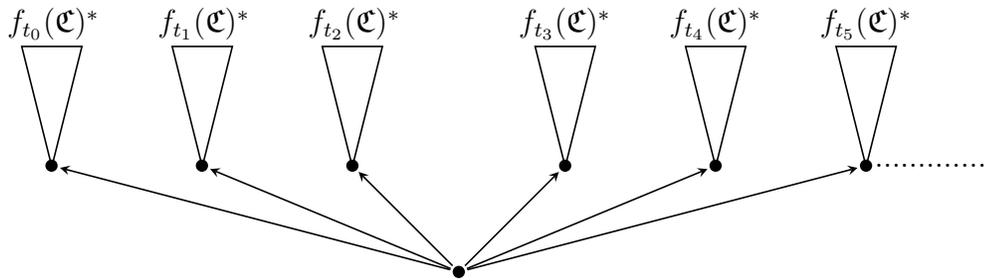
\begin{tikzpicture}[modal,node distance=2cm,world/.append style={minimum
size=1cm}]
\node[point] (root) {};

\node[point] (2) [above left of = root] {};

\node[point] (3) [above right of = root] {};

\node[point] (1) [left of = 2] {};

\node[point] (4) [right of = 3] {};

\node[point] (0) [left of = 1] {};

\node[point] (5) [right of = 4] {};

\path[->] (root) edge (0);

\path[->] (root) edge (1);

\path[->] (root) edge (2);

\path[->] (root) edge (3);

\path[->] (root) edge (4);

\path[->] (root) edge (5);

\draw (-1.415, 1.4)
  -- (-1.815,3) 
  -- (-1.015,3) 
  -- cycle;
  
\draw (-3.41, 1.4)
  -- (-3.81, 3) 
  -- (-3.01, 3) 
  -- cycle; 
  
\draw (-5.41, 1.4)
  -- (-5.81, 3) 
  -- (-5.01, 3) 
  -- cycle; 
  
\draw (1.415, 1.4)
  -- (1.015,3) 
  -- (1.815,3) 
  -- cycle; 
  
\draw (3.415, 1.4)
  -- (3.015,3) 
  -- (3.815,3) 
  -- cycle; 
  
\draw (5.415, 1.4)
  -- (5.015,3) 
  -- (5.815,3) 
  -- cycle; 
  
\draw (-5.9, 3.3) node[text width = 0.2cm] 
         {$f_{t_0}(\mathfrak{C})^*$}; 
         
\draw (-3.9, 3.3) node[text width = 0.2cm] 
         {$f_{t_1}(\mathfrak{C})^*$}; 
         
\draw (-1.9, 3.3) node[text width = 0.2cm] 
         {$f_{t_2}(\mathfrak{C})^*$}; 
         
\draw (0.9, 3.3) node[text width = 0.2cm] 
         {$f_{t_3}(\mathfrak{C})^*$}; 
         
\draw (2.9, 3.3) node[text width = 0.2cm] 
         {$f_{t_4}(\mathfrak{C})^*$}; 
         
\draw (4.9, 3.3) node[text width = 0.2cm] 
         {$f_{t_5}(\mathfrak{C})^*$}; 
         
\draw node[text width = 0cm, right of = 4]
	{..............};
\end{tikzpicture}
$$
where the root $w$ is a new world such that $|R^n|(w) = \bigcap_{t \in \mathcal{L}_0^+} |R^n|(f_t(\Gamma_t))$. By Intersection, $w \Vdash \Gamma$.  By Persistence, $w \not \Vdash \phi(t)$ for all $t \in \mathcal{L}_0^+$. So $w \not \Vdash \exists v \phi$. Taking the reduct of this model to $\mathcal{L}_0$ and then arbitrarily expanding to $\mathcal{L}$ gives $\Gamma \not \models \exists v \phi$.
\end{proof}

\begin{theorem}[Existence Property] Suppose $\mathcal{L}$ contains at least one constant symbol. Then, for $\Gamma \subseteq \mathcal{L} \setminus \{\rightarrow, \vee, \exists\}$: $\Gamma \models \exists v \phi$ only if $\Gamma \models \phi(t)$ for some $t \in \mathcal{L}$.
\end{theorem}
\begin{proof} Immediate from compactness and the weak existence property.
\end{proof}

\section{References}

[1] Gabbay D. M., Shehtman V. B., and Skvortsov D. P., \textit{Quantification in Non-Classical Logic: Volume 1}, Elsevier, 2009. 

[2] Ishigaki R., and Kikuchi K., Tree-Sequent Methods for Subintuitionistic Predicate Logics, in Olivetti N. (ed.), \textit{Automated Reasoning with Analytic Tableaux and Related Methods: 16th International Conference Proceedings}, Springer, 2007, pp. 149 -- 164.

[3] Poizat B., \textit{A Course in Model Theory: An Introduction to Contemporary Mathematical Logic}, Springer, 2000.

[4] Priest G., Paraconsistent Logic, in Gabbay D. M. and Guenthner F. (eds.), \textit{Handbook of Philosophical Logic, 2nd Edition: Volume 6}, Kluwer Academic Publishers, 2002, pp. 287 -- 393.

[5] Restall G., Subintuitionistic Logics, \textit{Notre Dame Journal of Formal Logic} 35(1): 116 -- 129, 1994.

[6] Ruitenburg W., Basic Predicate Calculus, \textit{Notre Dame Journal of Formal Logic} 39(1): 18 -- 46, 1998.

[7] Visser A., A Propositional Logic with Explicit Fixed Points, \textit{Studia Logica} 40(2): 155 -- 175, 1981. 

\end{document}